\documentclass[12pt]{amsart}%
\usepackage{bbding}
\usepackage{graphicx}
\usepackage{amsthm,amscd}
\usepackage[mathcal]{euscript}
\usepackage{calc}
\usepackage{hyperref}
\usepackage{mathrsfs}
\usepackage{CJK,fancyhdr,amscd}
\usepackage{anysize}
\usepackage{amsmath,amssymb,amsfonts}
\usepackage{epsfig}
\usepackage{latexsym}
\usepackage{hyperref}
\usepackage{indentfirst, latexsym}
\usepackage{graphics}
\usepackage{amsmath}
\usepackage{amsfonts}
\usepackage{amssymb}%
\usepackage[all]{xy}

\providecommand{\U}[1]{\protect\rule{.1in}{.1in}}

\numberwithin{equation}{section}
\providecommand{\U}[1]{\protect\rule{.1in}{.1in}}
\newtheorem{theorem} {Theorem} [section]
\newtheorem{proposition}[theorem]{Proposition}
\newtheorem{corollary}  [theorem]     {Corollary}

\begin{document}

\title{The correspondence formula of Dolbeault complex on pair deformation}

\author{Jie Tu}
\address{Corresponding author. Center of Mathematical Sciences, Zhejiang University, Hangzhou, 310027 China}
\email{tujie@zju.edu.cn, tujietujietujie@126.com}

\date{\today}

\subjclass[2010]{}

\begin{abstract}
Given a holomorphic family of pairs $\{(X_t,E_t)\}$, where each $E_t$ is holomorphic vector bundle over compact complex manifold $X_t$. For small enough $t$, we get a correspondence between the Dolbeault complex of $E_t$-valued $(p,q)$-forms on $X_t$ and the one of $E_0$-valued $(p,q)$-forms on $X_0$.
\end{abstract}
\maketitle

\section{Introduction}

The Beltrami differential( or Kuranishi data) plays a central role in analytic deformation theory. Given a holomorphic family of compact complex manifolds, i.e., a holomorphic proper submersion
\begin{align*}
\pi:\mathcal{X}\rightarrow\Delta
\end{align*}
from a total complex manifold $\mathcal{X}$ to a small neighbourhood $\Delta\subset\mathbb{C}^m$ of $0$. There exists a transversely holomorphic trivialization (see \cite[Appendix A]{Clemens2005}, \cite[Proposition 9.5]{Voisin2002} and \cite[Theorem IV.31]{Manetti2005})
\begin{align*}
F=(\sigma,\pi):\mathcal{X}\rightarrow X_0\times\Delta,\quad X_t:=\pi^{-1}(t),
\end{align*}
such that each fibre $X_t$ is considered as $X_0$ with a new complex structure $J_t$. The Beltrami differential is a holomorphic tangent bundle valued form
\begin{align*}
\varphi(t)\in A^{0,1}(X_0,T^{1,0}X_0)
\end{align*}
which describes the varying of complex structures on $X_0$.

In \cite{Liu-Rao-Yang2014}, Liu-Rao-Yang defined an extension map
\begin{align*}
e^{i_{\varphi(t)}}=\sum^{\infty}_{k=0}\frac{1}{k!}i^{k}_{\varphi(t)},\quad i^{k}_{\varphi(t)}=\underbrace{i_{\varphi(t)}\circ \cdots \circ i_{\varphi(t)}}_{k ~~~\text{times}}.
\end{align*}
which is a linear isomorphism
\begin{align*}
e^{i_{\varphi(t)}}:A^{n,p}(X_0)\rightarrow A^{n,p}(X_t)
\end{align*}
between $(n,q)$-forms on $X_0$ and $(n,q)$-forms on $X_t$, where $n=\dim_{\mathbb{C}}X_0$.

They also proved a formula \cite[Theorem 1.3]{Liu-Rao-Yang2014} that
\begin{align}\label{1-1}
e^{-i_{\varphi(t)}}\circ\nabla\circ e^{i_{\varphi(t)}}
=\nabla-\mathcal{L}^{1,0}_{\varphi(t)}+i_{\bar{\partial}\varphi(t)-\frac{1}{2}\left[\varphi(t),\varphi(t)\right]}
\end{align}
where $\nabla$ is a Chern connection, and the generalized Lie derivative is defined as
\begin{align*}
\mathcal{L}^{1,0}_{\varphi(t)}:=i_{\varphi(t)}\circ\nabla^{1,0}-\nabla^{1,0}\circ i_{\varphi(t)}.
\end{align*}

As a direct corollary \cite[Proposition 5.1]{Liu-Rao-Yang2014}, the Dolbeault operator $\bar{\partial}_t$ on the complex $A^{n,\cdot}(X_t)$ satisfies the following commutative diagram.
\begin{align*}
\xymatrix{
A^{n,q}(X_0) \ar[rr]^{e^{i_{\varphi(t)}}}\ar[d]_{\bar{\partial}_t}& & A^{n,q}(X_t) \ar[d]_{\bar{\partial}+\partial\circ i_{\varphi(t)}}\\
A^{n,q+1}(X_0) \ar[rr]^{e^{i_{\varphi(t)}}}& & A^{n,q+1}(X_t)
}
\end{align*}

Then, they calculated an extension formula from a holomorphic form $\sigma\in A^{n,0}(X_0)$ to a holomorphic form $\sigma(t)\in A^{n,0}(X_t)$ by solving the obstruction equation
\begin{align*}
\bar{\partial}\sigma(t)+\partial\circ i_{\varphi(t)}\sigma(t)=0
\end{align*}
iteratively.

Recently, the extension map was generalized by Rao-Zhao \cite{Rao-Zhao2017} into any $\bar{\partial}$-closed forms in $A^{p,q}(X_0)$. They defined a linear isomorphism
\begin{align*}
e^{i_\varphi|i_{\bar{\varphi}}}:A^{p,q}(X_0)\rightarrow A^{p,q}(X_t).
\end{align*}
In this situation, the $\bar{\partial}_t$-closed equation on $A^{p,q}(X_t)$ corresponds to the equation
\begin{align*}
\left(\bar{\partial}-L^{1,0}_{\varphi(t)}\right)\sigma(t)=0
\end{align*}
on $X_0$, where the holomorphic Lie derivative is defined as
\begin{align*}
L^{1,0}_{\varphi(t)}:=i_{\varphi(t)}\circ\partial-\partial\circ i_{\varphi(t)}.
\end{align*}

The main purpose of this article is to use the pair deformation $\{(X_t,E_t)\}$ which was described in \cite{Liu-Tu2018} to give a correspondence between $E_t$-valued $(p,q)$-forms on $X_t$ and $E_0$-valued $(p,q)$-forms on $X_0$.

The correspondence is defined by a commutative diagram
\begin{align*}
\xymatrix{
A^{p,q}(X_0,E_0) \ar[rr]^{I^{p,q}_0}\ar[d]_{P_t}& & A^{n,q}(X_0,\Omega_0^p\otimes K_0^{-1}\otimes E_0) \ar[d]_{e^{i_{\phi(t)}}}\\
A^{p,q}(X_t,E_t) \ar[rr]^{I^{p,q}_t}& &  A^{n,q}(X_t,\Omega_t^p\otimes K_t^{-1}\otimes E_t)
}
\end{align*}
where $I^{p,q}_t$ preserves the $\bar{\partial}_t$-complex.

The Dolbeault operator $\bar{\partial}_t$ on each fiber $\left(X_t,E_t\right)$ satisfies
\begin{theorem}(Theorem \ref{thm-2})
For any $E_0$-valued $(p,q)$-form $s\in A^0(X_0,E_0)$, we have
\begin{align*}
\bar{\partial}_t\circ P_t(s)=P_t\circ\left(\bar{\partial}-\mathcal{L}^{1,0}_{\phi(t)}+\psi_{E_t}\right)s£¬
\end{align*}
\end{theorem}

When each $E_t$ is the trivial line bundle, we get a correspondence between $(p,q)$-forms on $X_t$ and $(p,q)$-forms on $X_0$.
\begin{theorem}(Theorem \ref{thm-1})
For any $(p,q)$-form $s\in A^{p,q}(X_0)$, we have
\begin{align*}
\bar{\partial}_t\circ P_t (s)=P_t\circ\left(\bar{\partial}-L^{1,0}_{\phi(t)}\right)s.
\end{align*}
\end{theorem}
In this situation, the isomorphism map $P_t$ is different from that in \cite{Rao-Zhao2017}, but the correspondence of Dolbeault operator is the same.

The correspondence supplies a way to transform the Dolbeault complex on each fiber $X_t$ into a family of elliptic complexes on $X_0$. Then we can use standard techniques from the theory of elliptic operators to solve the varying cohomology problems. A direct application is to use iteration method to extend an element in cohomology class such as \cite[Theorem 5.5]{Liu-Rao-Yang2014}, \cite[Theorem 1.3, 1.4]{Rao-Zhao2017}, \cite[Theorem 1.5, 1.7]{Rao-Wan-Zhao2016} and \cite[Theorem 1.5]{Liu-Tu2018}. On the other aspect, the correspondence of Dolbeault operator can be applied to study the obstruction to the cohomology extension such as in \cite{Green-Lazarsfeld1987}, \cite{Green-Lazarsfeld1991} and \cite{Chan-Suen2018}.

\section{preliminary}

In this section, we will review some basic results of pair deformation in \cite{Liu-Tu2018}.

Let $X$ be a compact complex manifold and $p:E\rightarrow X$ be a holomorphic vector bundle over $X$, then the short exact sequence of tangent bundle
\begin{align*}
0\rightarrow\ker p_\ast\rightarrow TE\rightarrow p^\ast(TX)\rightarrow0
\end{align*}
has a smooth splitting via any linear connection $\nabla$ on $E$(see p.335 in \cite{Greub-Halperin-Vanstone1973}).

If the linear connection is integrable(i.e. $\nabla=\bar{\partial}+\nabla^{1,0}$), then it defines a smooth splitting on holomorphic tangent bundle
\begin{align*}
0\rightarrow\ker p_\ast|_{T^{1,0}E}\rightarrow T^{1,0}E\rightarrow p^\ast(T^{1,0}X)\rightarrow0.
\end{align*}

In \cite[Proposition 3.2]{Liu-Tu2018}, we use its dual case to get a smooth projection \begin{align*}
q_{\nabla}:T^\ast_{1,0}(E)\rightarrow p^\ast\left(T^\ast_{1,0}(X)\right)\quad \left.q_\nabla\right|_{p^\ast\left(T^\ast_{1,0}(X)\right)}=Id
\end{align*}
on holomorphic cotangent bundle of $E$ via an integrable connection $\nabla$ .

Precisely, under the holomorphic frame coordinate $(U_\alpha,e_\alpha,z_\alpha,v_\alpha)$ of $E$ (i.e., $e_\alpha$ is local frame of $E$ on an open set $U_\alpha\subset X$, $z_\alpha$ are holomorphic coordinate functions of $X$, and $v_\alpha$ are coordinate functions of vertical linear space), the projection $q_{\nabla}$ can be represented locally as
\begin{align*}
dz^i_\alpha\mapsto dz^i_\alpha,\quad dv^i_\alpha\mapsto-\theta_\alpha^{ki}v^k_\alpha,
\end{align*}
where $\theta_\alpha$ is the connection 1-form of $\nabla$ on $E$.

Recall that a holomorphic family of pairs is a holomorphic vector bundle
\begin{align*}
\mathcal{E}\overset{p}{\rightarrow}\mathcal{X}\overset{\pi}{\rightarrow}\Delta
\end{align*}
over total space $\mathcal{X}$.

As \cite[Proposition 4.1]{Liu-Tu2018}, when $\Delta$ is small enough, there is a compatible trivialization
\begin{align*}
\xymatrix{
  \mathcal{E} \ar[d] \ar[r]^{(\lambda,\pi\circ p)}
                & E_0\times\Delta \ar[d] \ar[r]^{pr2} & \Delta\ar[d]^{Id} \\
  \mathcal{X}\ar[r]^{(\sigma,\pi)}
                & X_0\times\Delta \ar[r]^{pr2} & \Delta            }
\end{align*}
such that on each fibre $\left(X_t,E_t\right)$ we have the following commutative diagram
\begin{align*}
\xymatrix{
\lambda_t:&E_t \ar[d] \ar[r] &\left(\sigma_t^{-1}\right)^\ast E_t \ar[d] \ar[r] & E_0 \ar[d] \\
 &X_t \ar[r]^{\sigma_t} & X_0 \ar[r]^{Id} & X_0   }
\end{align*}
where
\begin{align*}
\sigma_t:=\left.\sigma\right|_{X_t}\quad\text{and}\quad\lambda_t:=\left.\lambda\right|_{E_t}.
\end{align*}

Given a holomorphic family of pairs $\{(X_t,E_t)\}$, the Beltrami differential
\begin{align*}
\xi(t):T^\ast_{1,0}(E_0)\rightarrow T^\ast_{0,1}(E_0)
\end{align*}
can be represented as two parts
\begin{align*}
\phi(t)\in A^{0,1}(X_0,T^{1,0}X_0)
\end{align*}
and
\begin{align*}
\psi(t)\in A^{0,1}\left(X_0,End(E_0)\right)
\end{align*}
by \cite[Theorem 4.3]{Liu-Tu2018}. The projection $q_\nabla$ via an integrable connection $\nabla$ decides the way to decompose $\xi(t)$ into $\phi(t)$ and $\psi(t)$.

To be specific, if we have a transversely holomorphic trivialization $\sigma_t:X_t\rightarrow X_0$ and its compatible trivialization $\lambda_t:E_t\rightarrow E_0$, then we get the local representations
\begin{align*}
\phi(t)=\left(\frac{\partial(z_\alpha)_t}{\partial(z_\alpha)_0}\right)^{-1}_{ik}
\frac{\partial(z_\alpha)^k_t}{\partial(\bar{z}_\alpha)^j_0}d(\bar{z}_\alpha)^j_0\otimes\frac{\partial}{\partial(z_\alpha)^i_0}
\end{align*}
and
\begin{align}\label{psi}
\psi_\alpha^{kl}(t)=(w^{-1}_\alpha)^{lj}\left((\bar{\partial}-L_{\phi(t)})w_\alpha^{jk}\right)+i_{\phi(t)}\theta_\alpha^{kl},
\end{align}
where
\begin{align}\label{com-tri}
(v_\alpha)^j_t=w^{jl}_\alpha(v_\alpha)^l_0
\end{align}
under the holomorphic frame coordinate $(U_\alpha,(e_\alpha)_t,(z_\alpha)_t,(v_\alpha)_t)$ of $(X_t,E_t)$.

In \cite[Theorem 4.9]{Liu-Tu2018}, the decomposition $\xi(t)=\phi(t)+\psi(t)$ via a Chern connection divides the Maurer-Cartan equation
\begin{align*}
\bar{\partial}\xi(t)=\frac{1}{2}\left[\xi(t),\xi(t)\right]
\end{align*}
into two equations
\begin{align}\label{M-C}
\bar{\partial}\phi(t)=\frac{1}{2}\left[\phi(t),\phi(t)\right]
\end{align}
and
\begin{align}\label{2-9}
\left(\bar{\partial}-\mathcal{L}_{\phi(t)}\right)\psi(t)-\psi(t)\wedge\psi(t)-i_{\phi(t)}\Theta=0,
\end{align}
where
\begin{align*}
\mathcal{L}_{\phi(t)}=i_{\phi(t)}\circ\nabla-\nabla\circ i_{\phi(t)}.
\end{align*}

We introduce the following two propositions which will be used in later sections.

\begin{proposition}\label{pro-3}
Given a holomorphic pair deformation $\{(X_t,E_t)\}$, if the decomposition
\begin{align*}
\xi(t)=\phi(t)+\psi(t)
\end{align*}
of Beltrami differential is via an integrable connection $\nabla$ on $E_{0}$, then $\nabla+\psi(t)$ is an integrable connection on $E_{t}$.
\end{proposition}

\begin{proof}
Under the frame coordinate $\left(U_{\alpha},(e_{\alpha})_{t},(x_{\alpha})_{t},(v_{\alpha})_{t}\right)$ of $(X_t,E_{t})$, the smooth transition between two holomorphic basis is
\begin{align*}
(e_{\alpha})_{t}^{k}=(w_{\alpha}^{-1})^{lk}(e_{\alpha})_{0}^{l}
\end{align*}
as (\ref{com-tri}).

Then, by the representation of $\psi(t)$ in (\ref{psi}) we have
\begin{align*}
\nabla(e_{\alpha})_{t}^{k}&=d(w_{\alpha}^{-1})^{lk}(e_{\alpha})_{0}^{l}+(w_{\alpha}^{-1})^{lk}\theta_{\alpha}^{lj}(e_{\alpha})_{0}^{j}\\
&=\left(\partial(w_{\alpha}^{-1})^{jk}+(w_{\alpha}^{-1})^{lk}\theta_{\alpha}^{lj}\right)(e_{\alpha})_{0}^{j}
+i_{\phi(t)}\left(\partial(w_{\alpha}^{-1})^{jk}+(w_{\alpha}^{-1})^{lk}\theta_{\alpha}^{lj}\right)(e_{\alpha})_{0}^{j}\\
&\quad\quad+\left((\bar{\partial}-L_{\phi(t)})(w_{\alpha}^{-1})^{jk}-(w_{\alpha}^{-1})^{lk}i_{\phi(t)}\theta_{\alpha}^{lj}\right)(e_{\alpha})_{0}^{j}\\
&=\left(\partial(w_{\alpha}^{-1})^{jk}+(w_{\alpha}^{-1})^{lk}\theta_{\alpha}^{lj}\right)(e_{\alpha})_{0}^{j}
+i_{\phi(t)}\left(\partial(w_{\alpha}^{-1})^{jk}+(w_{\alpha}^{-1})^{lk}\theta_{\alpha}^{lj}\right)(e_{\alpha})_{0}^{j}\\
&\quad\quad-(w_{\alpha}^{-1})^{lk}\psi_{\alpha}^{lj}(t)(e_{\alpha})_{0}^{j}\\
&=\alpha+i_{\phi(t)}\alpha-\psi(t)((e_{\alpha})_{t}^{k}),
\end{align*}
where $\alpha+i_{\phi(t)}\alpha$ is bundle valued $(1,0)$-form on $X_t$. Hence, $\nabla+\psi(t)$ is an integrable connection on $(E_t,X_t)$.
\end{proof}

It was proved in Section 5 of \cite{Liu-Tu2018}, that the operator $e^{i_{\phi(t)}}$ also defines a linear isomorphism
\begin{align*}
e^{i_{\phi(t)}}:A^{n,q}(X_0,E_0)\rightarrow A^{n,q}(X_t,E_t).
\end{align*}
between $E_0$-valued $(n,q)$-forms on $X_0$ and $E_t$-valued $(n,q)$-forms on $X_t$.

Moreover, we have

\begin{proposition}\label{pro-2}
Given any $\sigma\in A^{n,q}(X_0,E_0)$, then $e^{i_{\phi(t)}}\sigma \in A^{n,q}(X_{t},E_t)$ and
\begin{align*}
e^{-i_\phi(t)}\circ\bar{\partial}_t\circ e^{i_\phi(t)}(\sigma)
=\bar{\partial}\sigma+\nabla^{1,0}_{E_0}\circ i_{\phi(t)}(\sigma)+\psi_{E_t}\sigma.
\end{align*}
\end{proposition}

\begin{proof}
By Proposition \ref{pro-3}, $\nabla+\psi_{E_t}$ is an integrable connection of $(X_t,E_t)$, thus
\begin{align*}
e^{-i_\phi(t)}\circ\left(\nabla+\psi_{E_t}\right)\circ e^{i_\phi(t)}(\sigma)=e^{-i_\phi(t)}\circ\bar{\partial}_t\circ e^{i_\phi(t)}(\sigma)
\end{align*}
for any $\sigma\in A^{n,q}(X_0,E_0)$.

Because $\phi(t)$ satisfies Maurer-Cartan equation (\ref{M-C}), (\ref{1-1}) becomes
\begin{equation}
e^{-i_{\phi(t)}}\circ\nabla\circ e^{i_{\phi(t)}}=\nabla-\mathcal{L}_{\phi(t)}^{1,0}.
\end{equation}
In particular, when $\sigma\in A^{n,q}(X_0,E_0)$, we have
\begin{equation}\label{2-8}
\left(e^{-i_{\phi(t)}}\circ\nabla\circ e^{i_{\phi(t)}}\right)\sigma=\bar{\partial}\sigma+\nabla^{1,0}_{E_0}\circ i_{\phi(t)}\sigma.
\end{equation}

On the other hand, $\psi_{E_t}$ is of type $(0,1)$, which implies
\begin{align*}
i_{\phi(t)}\circ\psi_{E_t}=\psi_{E_t}\circ i_{\phi(t)}.
\end{align*}

Hence, we have
\begin{align*}
e^{-i_\phi(t)}\circ\bar{\partial}_t\circ e^{i_\phi(t)}(\sigma)
=e^{-i_\phi(t)}\circ\left(\nabla+\psi_{E_t}\right)\circ e^{i_\phi(t)}(\sigma)
=\bar{\partial}\sigma+\nabla^{1,0}_{E_0}\circ i_{\phi(t)}(\sigma)+\psi_{E_t}\sigma.
\end{align*}
\end{proof}

\section{The Correspondence between $(p,q)$-forms}

In this section, we will use the compatible trivialization of holomorphic cotangent bundle to get a correspondence between $\bar{\partial}_t$-complex of $(p,q)$-form on $X_t$ and $\bar{\partial}_0$-complex of $(p,q)$-forms on $X_0$.

Recall that the compatible trivialization of pairs $\{(X_t,\Omega_t)\}$ in Section 6 of \cite{Liu-Tu2018} is
\begin{align*}
\pi^{1,0}:\Omega_t\rightarrow\Omega_0,\quad d(z_{\alpha})_{t}^{k}\mapsto\frac{\partial(z_{\alpha})_{t}^{k}}{\partial(z_{\alpha})_{0}^{j}}d(z_{\alpha})_{0}^{j},
\end{align*}
where $\Omega_t$ is the holomorphic cotangent bundle of $X_t$, and $(z_{\alpha})_{t}$ is the holomorphic local coordinate of $X_t$.

Suppose that $\phi(t)$ is the Beltrami differential of $X_t$ and $(z^i)$ is a fixed holomorphic coordinate of $X_0$, then we proved in Section 6 of \cite{Liu-Tu2018} that the decomposition of Beltrami differential
\begin{align*}
\xi_{(X_t,\Omega_t)}=\phi(t)+\psi_{\Omega_t}
\end{align*}
via a Chern connection $\nabla$ satisfies
\begin{align*}
\psi^{kl}_{\Omega_t}=-\partial_l\phi^k-\Gamma^k_{jl}\phi^j,\quad
\phi=\phi^i_jd\bar{z}^j\otimes\frac{\partial}{\partial z_i}=\phi^i\frac{\partial}{\partial z_i},
\end{align*}
where
\begin{align*}
\nabla_{\frac{\partial}{\partial z_i}}\frac{\partial}{\partial z_j}=\Gamma^k_{ij}\frac{\partial}{\partial z_k}.
\end{align*}

Suppose that $\Omega^p_t$ is the $p$-times wedge product of holomorphic cotangent bundle $\Omega_t$, and $K_t$ is the canonical bundle of $X_t$.

\begin{proposition}
For the pair deformations $\{(X_t,\Omega^p_t)\}$ and $\{(X_t,K_t)\}$ we have
\begin{align*}
\psi_{\Omega^p_t}\left(dz^I\right)
=\partial\circ i_\phi dz^I+i_\phi\circ\nabla^{1,0}_{\Omega^p_0}(dz^I)
\end{align*}
and
\begin{align*}
\psi_{K^{-1}_t}\left(\frac{1}{\langle dz\rangle}\right)
=\partial_l\phi^l_md\bar{z}^m\left(\frac{1}{\langle dz\rangle}\right)+i_\phi\circ\nabla^{1,0}_{K^{-1}_0}\left(\frac{1}{\langle dz\rangle}\right),
\end{align*}
where
\begin{align*}
\langle dz\rangle=dz^1\wedge\cdots\wedge dz^n.
\end{align*}
\end{proposition}

\begin{proof}
By \cite[Corollary 4.4]{Liu-Tu2018}, we have
\begin{align*}
\psi_{\Omega^p_t}\left(dz^I\right)
&=\sum_k(-1)^{k-1}dz^{i_1}\wedge\cdots\wedge\psi_{\Omega_t}\left(dz^{i_k}\right)\wedge\cdots\wedge dz^{i_p}\\
&=\sum_{k}(-1)^{k-1}dz^{i_1}\wedge\cdots\wedge\left(-\partial_l\phi^{i_k}_{m} d\bar{z}^m
-\Gamma^{i_k}_{jl}\phi^j_md\bar{z}^m\right)\wedge dz^l\wedge\cdots\wedge dz^{i_p}\\
&=\partial\circ i_\phi dz^I+i_\phi\circ\nabla^{1,0}_{\Omega^p_0}(dz^I),
\end{align*}
and
\begin{align*}
\psi_{K^{-1}_t}\left(\frac{1}{\langle dz\rangle}\right)
&=\left(\partial_l\phi^l_md\bar{z}^m+\Gamma^l_{jl}\phi^j_{m}d\bar{z}^m\right)\otimes\frac{1}{\langle dz\rangle}\\
&=\partial_l\phi^l_md\bar{z}^m\left(\frac{1}{\langle dz\rangle}\right)+i_\phi\circ\nabla^{1,0}_{K^{-1}_0}\left(\frac{1}{\langle dz\rangle}\right).
\end{align*}
\end{proof}

Moreover, for the pair deformation $\{(X_t,\Omega_t^p\otimes K_t^{-1})\}$ we have
\begin{corollary}
\begin{align*}
&\psi_{\Omega^p_t\otimes K^{-1}_t}\left(dz^I\otimes\frac{1}{\langle dz\rangle}\right)\\
&=\psi_{\Omega^p_t}\left(dz^I\right)\otimes\frac{1}{\langle dz\rangle}
+dz^I\otimes\psi_{K^{-1}_t}\left(\frac{1}{\langle dz\rangle}\right)\\
&=\partial\circ i_\phi dz^I\otimes\frac{1}{\langle dz\rangle}
+\partial_l\phi^l_md\bar{z}^m\otimes dz^I\otimes\frac{1}{\langle dz\rangle}
+i_{\phi(t)}\circ\nabla^{1,0}_{\Omega^p_0\otimes K^{-1}_0}\left(dz^I\otimes\frac{1}{\langle dz\rangle}\right).
\end{align*}
\end{corollary}

At first, we need to establish the commutative diagram
\begin{align*}
\xymatrix{
A^{p,q}(X_t) \ar[rr]^{I^{p,q}_t}\ar[d]_{\bar{\partial}_t}& & A^{n,q}(X_t,\Omega_t^p\otimes K_t^{-1}) \ar[d]_{\bar{\partial}_t}\\
A^{p,q+1}(X_t) \ar[rr]^{I^{p,q+1}_t}& &  A^{n,q+1}(X_t,\Omega_t^p\otimes K_t^{-1}).
}
\end{align*}

\begin{proposition}
There is an isomorphism $I^{p,q}_t$ such that the above diagram is commutative.
\end{proposition}

\begin{proof}
Given any $t\in\Delta$, fix a holomorphic coordinate $(z^i)$ on $X_t$, then a $(p,q)$-form $s$ on $X_t$ is represented locally as
\begin{align*}
s=\sum_{I,J}s_{I,\bar{J}}d\bar{z}^J\wedge dz^I,\quad dz^I=dz^{i_1}\wedge\cdots\wedge dz^{i_p},\quad d\bar{z}^J=d\bar{z}^{j_1}\wedge\cdots\wedge d\bar{z}^{j_q}.
\end{align*}
the operator $I^{p,q}$ is defined as
\begin{align*}
I^{p,q}s:=\sum_{I,J}s_{I,\bar{J}}d\bar{z}^J\wedge\langle dz\rangle\otimes dz^I\otimes\frac{1}{\langle dz\rangle},
\quad \langle dz\rangle=dz^1\wedge\cdots\wedge dz^n.
\end{align*}

Hence,
\begin{align*}
I^{p,q+1}\circ\bar{\partial}_t s
&=I^{p,q+1}\left(\sum_{I,J,k}\frac{\partial s_{I,\bar{J}}}{\partial\bar{z}^k}d\bar{z}_k\wedge d\bar{z}^J\wedge dz^I\right)\\
&=\sum_{I,J,k}\frac{\partial s_{I,\bar{J}}}{\partial\bar{z}^k}d\bar{z}_k\wedge d\bar{z}^J\wedge\langle dz\rangle\otimes dz^I\otimes\frac{1}{\langle dz\rangle}\\
&=\bar{\partial}_t\circ I^{p,q} s.
\end{align*}
\end{proof}

Then, denote
\begin{align*}
P_t=\left(I^{p,q}_t\right)^{-1}\circ e^{i_{\phi(t)}}\circ I^{p,q}_0,
\end{align*}
the diagram
\begin{align*}
\xymatrix{
A^{p,q}(X_0) \ar[rr]^{I^{p,q}_0}\ar[d]_{P_t}& & A^{n,q}(X_0,\Omega_0^p\otimes K_0^{-1}) \ar[d]_{e^{i_{\phi(t)}}}\\
A^{p,q}(X_t) \ar[rr]^{I^{p,q}_t}& &  A^{n,q}(X_t,\Omega_t^p\otimes K_t^{-1})
}
\end{align*}
is commutative.

\begin{theorem}\label{thm-1}
For any $(p,q)$-form $s\in A^{p,q}(X_0)$, we have
\begin{align*}
\bar{\partial}_t\circ P_t (s)=P_t\circ\left(\bar{\partial}_0-L^{1,0}_{\phi(t)}\right)s.
\end{align*}
\end{theorem}

\begin{proof}
By Proposition \ref{pro-2},
\begin{align*}
\bar{\partial}_t\circ P_t(s)
&=\bar{\partial}_t\circ\left(I^{p,q}_t\right)^{-1}\circ e^{i_{\phi(t)}}\circ I^{p,q}_0(s)\\
&=\left(I^{p,q+1}_t\right)^{-1}\circ\bar{\partial}_t\circ e^{i_{\phi(t)}}\circ I^{p,q}_0(s)\\
&=\left(I^{p,q+1}_t\right)^{-1}\circ e^{i_{\phi(t)}}\circ\left(\bar{\partial}_0+\nabla^{1,0}_{\Omega^p_0\otimes K^{-1}_0}\circ i_{\phi(t)}+\psi_{\Omega^p_t\otimes K^{-1}_t}\right)\circ I^{p,q}_0(s).
\end{align*}

It suffices to show that
\begin{align*}
\left(\nabla^{1,0}_{\Omega^p_0\otimes K^{-1}_0}\circ i_{\phi(t)}+\psi_{\Omega^p_t\otimes K^{-1}_t}\right)\circ I^{p,q}_0(s)
=I^{p,q+1}_0\circ\left(-L^{1,0}_{\phi(t)}\right)(s).
\end{align*}

Denote
\begin{align*}
\sigma=I^{p,q}_0s\in A^{n,q}(X_0,\Omega^p_0\otimes K_0^{-1}),
\end{align*}
and
\begin{align*}
\sigma=\sum_{I}\sigma_I\otimes dz^I\otimes\frac{1}{\langle dz\rangle},\quad \sigma_I=\sum_{J}s_{I,\bar{J}}d\bar{z}^J\wedge\langle dz\rangle.
\end{align*}

Then
\begin{align*}
&\nabla^{1,0}_{\Omega^p_0\otimes K^{-1}_0}\circ i_{\phi(t)}\sigma+\psi_{\Omega^p_t\otimes K^{-1}_t}\sigma\\
&=\left(\partial\circ i_{\phi(t)}\sigma_I\right)\otimes dz^I\otimes\frac{1}{\langle dz\rangle}
+(-1)^{n+q}\left(i_{\phi(t)}\sigma_I\right)\wedge\nabla^{1,0}_{\Omega^p_0\otimes K^{-1}_0}\left(dz^I\otimes\frac{1}{\langle dz\rangle}\right)\\
&\quad+(-1)^{n+q}\sigma_I\wedge\psi_{\Omega^p_t\otimes K^{-1}_t}\left(dz^I\otimes\frac{1}{\langle dz\rangle}\right)\\
&=\left(\partial\circ i_{\phi(t)}\sigma_I\right)\otimes dz^I\otimes\frac{1}{\langle dz\rangle}
+(-1)^{n+q}\left(i_{\phi(t)}\sigma_I\right)\wedge\nabla^{1,0}_{\Omega^p_0\otimes K^{-1}_0}\left(dz^I\otimes\frac{1}{\langle dz\rangle}\right)\\
&\quad+(-1)^{n+q}\sigma_I\otimes\partial\circ i_\phi dz^I\otimes\frac{1}{\langle dz\rangle}
+(-1)^{n+q}\sigma_I\wedge\sum_{m,l}\partial_l\phi^l_md\bar{z}^m\otimes dz^I\otimes\frac{1}{\langle dz\rangle}\\
&\quad+(-1)^{n+q}\sigma_I\wedge i_{\phi(t)}\nabla^{1,0}_{\Omega^p_0\otimes K^{-1}_0}\left(dz^I\otimes\frac{1}{\langle dz\rangle}\right).
\end{align*}

Moreover,
\begin{align*}
\left(i_\phi\langle dz\rangle\right)\wedge dz^k
&=\sum_i\phi^i_{\bar{j}}d\bar{z}^j\wedge(-1)^{i-1}dz^1\wedge\cdots\wedge\widehat{dz^i}\wedge\cdots\wedge dz^n\wedge dz^k\\
&=\phi^k_{\bar{j}}d\bar{z}^j\wedge(-1)^{k-1+n-k}\langle dz\rangle\\
&=-\langle dz\rangle\wedge i_\phi d\bar{z}^k
\end{align*}
implies
\begin{align*}
\left(i_{\phi(t)}\sigma_I\right)\wedge\nabla^{1,0}_{\Omega^p_0\otimes K^{-1}_0}\left(dz^I\otimes\frac{1}{\langle dz\rangle}\right)
=-\sigma_I\wedge i_{\phi(t)}\nabla^{1,0}_{\Omega^p_0\otimes K^{-1}_0}\left(dz^I\otimes\frac{1}{\langle dz\rangle}\right).
\end{align*}

On the other hand,
\begin{align*}
\partial\circ i_{\phi(t)}\sigma_I
&=\partial\left(\sum_{J}s_{I,\bar{J}}d\bar{z}^J\wedge i_{\phi(t)}\langle dz\rangle\right)\\
&=\sum_{J}\partial s_{I,\bar{J}}\wedge d\bar{z}^J\wedge i_{\phi(t)}\langle dz\rangle
+(-1)^q\sum_{J}s_{I,\bar{J}}d\bar{z}^J\wedge \partial (i_{\phi(t)}\langle dz\rangle)\\
&=(-1)^{n+q}\sum_{J} d\bar{z}^J\wedge i_{\phi(t)}\langle dz\rangle\wedge\partial s_{I,\bar{J}}
+(-1)^q\sum_{J}s_{I,\bar{J}}d\bar{z}^J\wedge \partial (i_{\phi(t)}\langle dz\rangle)\\
&=(-1)^{n+q+1}\sum_{J} d\bar{z}^J\wedge\langle dz\rangle\wedge\left(i_{\phi(t)}\circ\partial s_{I,\bar{J}}\right)
+(-1)^{q+1}\sum_{J,l,m}s_{I,\bar{J}}d\bar{z}^J\wedge \partial_l\phi^l_md\bar{z}^m\wedge\langle dz\rangle\\
&=(-1)^{q+1}\sum_{J}d\bar{z}^J\wedge\left(i_{\phi(t)}\circ\partial s_{I,\bar{J}}\right)\wedge\langle dz\rangle
+(-1)^{n+q+1}\sum_{J,l,m}\sigma_I\wedge\partial_l\phi^l_md\bar{z}^m.
\end{align*}

Hence,
\begin{align*}
&\nabla^{1,0}_{\Omega^p_0\otimes K^{-1}_0}\circ i_{\phi(t)}\sigma+\psi_{\Omega^p_t\otimes K^{-1}_t}\sigma\\
&=(-1)^{q+1}\sum_{J}d\bar{z}^J\wedge\left(i_{\phi(t)}\circ\partial s_{I,\bar{J}}\right)
\wedge\langle dz\rangle\otimes dz^I\otimes\frac{1}{\langle dz\rangle}
+\sum_I(-1)^{n+q}\sigma_I\otimes\partial\circ i_{\phi(t)} dz^I\otimes\frac{1}{\langle dz\rangle}\\
&=I^{p,q+1}_0\left((-1)^{q+1}d\bar{z}^J\wedge i_{\phi(t)}\circ\partial s_{I,\bar{J}}\wedge dz^I
+(-1)^qs_{I,\bar{J}}\wedge d\bar{z}^J\wedge\partial\circ i_{\phi(t)} dz^I\right)\\
&=I^{p,q+1}_0\left(-i_{\phi(t)}\circ\partial s_{I,\bar{J}}\wedge d\bar{z}^J\wedge dz^I
+(-1)^qs_{I,\bar{J}}\wedge d\bar{z}^J\wedge\partial\circ i_{\phi(t)} dz^I\right)\\
&=I^{p,q+1}_0\left(-L^{1,0}_{\phi(t)} s\right).
\end{align*}
The last equality follows from
\begin{align*}
\left(-L^{1,0}_{\phi(t)}\right)s
&=\left(\partial\circ i_{\phi(t)}-i_{\phi(t)}\circ\partial\right)s\\
&=\partial s_{I,\bar{J}}\wedge d\bar{z}^J\wedge i_{\phi(t)} dz^I
+(-1)^qs_{I,\bar{J}}\wedge d\bar{z}^J\wedge\partial\circ i_{\phi(t)} dz^I\\
&\quad-i_{\phi(t)}\circ\partial s_{I,\bar{J}}\wedge d\bar{z}^J\wedge dz^I
-\partial s_{I,\bar{J}}\wedge d\bar{z}^J\wedge i_{\phi(t)} dz^I\\
&=(-1)^qs_{I,\bar{J}}\wedge d\bar{z}^J\wedge\partial\circ i_{\phi(t)} dz^I
-i_{\phi(t)}\circ\partial s_{I,\bar{J}}\wedge d\bar{z}^J\wedge dz^I.
\end{align*}
\end{proof}

\section{The Correspondence between bundle valued $(p,q)$-forms}

Similar to above section, we need to define the isomorphism $I^{p,q}_t$ at first.

\begin{proposition}
There is an isomorphism $I^{p,q}_t$ such that the diagram
\begin{align*}
\xymatrix{
A^{p,q}(X_t,E_t) \ar[rr]^{I^{p,q}_t}\ar[d]_{\bar{\partial}_t}& & A^{n,q}(X_t,\Omega_t^p\otimes K_t^{-1}\otimes E_t) \ar[d]_{\bar{\partial}_t}\\
A^{p,q+1}(X_t,E_t) \ar[rr]^{I^{p,q+1}_t}& &  A^{n,q+1}(X_t,\Omega_t^p\otimes K_t^{-1}\otimes E_t)
}
\end{align*}
is commutative.
\end{proposition}

\begin{proof}
Fix a holomorphic coordinate $(z^i)$ on $X_t$, and holomorphic frame $\{e_k\}$ on $E_t$, any $E_t$-valued $(p,q)$-form $s$ on $X_t$ is represented as
\begin{align*}
s=\sum_{I,J,k}s_{I,\bar{J},k}d\bar{z}^J\wedge dz^I\otimes e_k,\quad dz^I=dz^{i_1}\wedge\cdots\wedge dz^{i_p},\quad d\bar{z}^J=d\bar{z}^{j_1}\wedge\cdots\wedge d\bar{z}^{j_q}.
\end{align*}
the operator $I^{p,q}$ is defined as
\begin{align*}
I^{p,q}s:=\sum_{I,J,k}s_{I,\bar{J},k}d\bar{z}^J\wedge\langle dz\rangle\otimes dz^I\otimes\frac{1}{\langle dz\rangle}\otimes e_k,
\quad \langle dz\rangle=dz^1\wedge\cdots\wedge dz^n.
\end{align*}
then
\begin{align*}
I^{p,q+1}\circ\bar{\partial}s
&=I^{p,q+1}\left(\sum_{I,J,k,l}\bar{\partial}_ls_{I,\bar{J},k}d\bar{z}_l\wedge d\bar{z}^J\wedge dz^I\otimes e_k\right)\\
&=\sum_{I,J,k,l}\bar{\partial}_ls_{I,\bar{J},k}d\bar{z}_l\wedge d\bar{z}^J\wedge\langle dz\rangle\otimes dz^I\otimes\frac{1}{\langle dz\rangle}\otimes e_k\\
&=\bar{\partial}\circ I^{p,q} s.
\end{align*}
\end{proof}

\begin{theorem}\label{thm-2}
Denote
\begin{align*}
P_t=\left(I^{p,q}_t\right)^{-1}\circ e^{i_{\phi(t)}}\circ I^{p,q}_0,
\end{align*}
then for any $E_0$-valued $(p,q)$-form $s\in A^0(X_0,E_0)$ we have
\begin{align*}
\bar{\partial}_t\circ P_t(s)=P_t\circ\left(\bar{\partial}-\mathcal{L}^{1,0}_{\phi(t)}+\psi_{E_t}\right)s£¬
\end{align*}
\end{theorem}

\begin{proof}
By Proposition \ref{pro-2},
\begin{align*}
\bar{\partial}_t\circ P_t(s)
&=\bar{\partial}_t\circ\left(I^{p,q}_t\right)^{-1}\circ e^{i_{\phi(t)}}\circ I^{p,q}_0(s)\\
&=\left(I^{p,q+1}_t\right)^{-1}\circ\bar{\partial}_t\circ e^{i_{\phi(t)}}\circ I^{p,q}_0(s)\\
&=\left(I^{p,q+1}_t\right)^{-1}\circ e^{i_{\phi(t)}}\circ\left(\bar{\partial}_0+\nabla^{1,0}_{\Omega^p_0\otimes K^{-1}_0\otimes E_0}\circ i_{\phi(t)}+\psi_{\Omega^p_t\otimes K^{-1}_t\otimes E_0}\right)\circ I^{p,q}_0(s).
\end{align*}

Thus, we just need to prove
\begin{align*}
\left(\nabla^{1,0}_{\Omega^p_0\otimes K^{-1}_0\otimes E_0}\circ i_{\phi(t)}+\psi_{\Omega^p_t\otimes K^{-1}_t\otimes E_0}\right)\circ I^{p,q}_0(s)=I^{p,q+1}_0\circ\left(-\mathcal{L}^{1,0}_{\phi(t)}+\psi_{E_t}\right)s.
\end{align*}

Denote
\begin{align*}
\sigma=I^{p,q}_0s\in A^{n,q}(X_0,\Omega^p_0\otimes K_0^{-1}\otimes E_0)
\end{align*}
and
\begin{align*}
\sigma=\sum_{I,k}\sigma_{I,k}\otimes dz^I\otimes\frac{1}{\langle dz\rangle}\otimes e_k,\quad \sigma_{I,k}=\sum_{J}s_{I,\bar{J},k}d\bar{z}^J\wedge\langle dz\rangle.
\end{align*}

Then, we have
\begin{align*}
&\nabla^{1,0}_{\Omega^p_0\otimes K^{-1}_0\otimes E_0}\circ i_{\phi(t)}\sigma
+\psi_{\Omega^p_t\otimes K^{-1}_t\otimes E_t}\sigma\\
&=\nabla^{1,0}_{\Omega^p_0\otimes K^{-1}_0}\circ i_{\phi(t)}\left(\sigma_{I,k}\otimes dz^I\otimes\frac{1}{\langle dz\rangle}\right)\otimes e_k
+(-1)^{n+q}\left(i_{\phi(t)}\sigma_{I,k}\otimes dz^I\otimes\frac{1}{\langle dz\rangle}\right)\otimes\nabla^{1,0}_{E_0}e_k\\
&\quad+\psi_{\Omega^p_t\otimes K^{-1}_t}\left(\sigma_{I,k}\otimes dz^I\otimes\frac{1}{\langle dz\rangle}\right)\otimes e_k+(-1)^{n+q}\sigma_{I,k}\otimes dz^I\otimes\frac{1}{\langle dz\rangle}\otimes\psi_{E_t}e_k.
\end{align*}
On the other hand,
\begin{align*}
\mathcal{L}^{1,0}_\phi s &=\left(i_\phi\circ\nabla^{1,0}_{E_0}-\nabla^{1,0}_{E_0}\circ i_\phi\right)s\\
&=\sum_k i_\phi\left(\partial s_k\otimes e_k+(-1)^{p+q} s_k\wedge\nabla^{1,0}_{E_0} e_k\right)
-\sum_k\nabla^{1,0}_{E_0}\left(i_\phi s_k\otimes e_k\right)\\
&=\sum_k L^{1,0}_\phi s_k\otimes e_k+\sum_k(-1)^{p+q}s_k\wedge i_\phi\nabla^{1,0}_{E_0} e_k.
\end{align*}
Hence, by Theorem \ref{thm-1}
\begin{align*}
&\nabla^{1,0}_{\Omega^p_0\otimes K^{-1}_0\otimes E_0}\circ i_{\phi(t)}\sigma+\psi_{\Omega^p_t\otimes K^{-1}_t\otimes E_t}\sigma\\
&=I^{p,q+1}_0(-L^{1,0}_{\phi(t)}s_k\otimes e_k-\sum_k(-1)^{p+q}s_k\wedge i_\phi\nabla^{1,0}_{E_0} e_k
+\psi_{E_t}s)\\
&=I^{p,q+1}_0\circ\left(-\mathcal{L}^{1,0}_\phi+\psi_{E_t}\right)s
\end{align*}
\end{proof}

Theorem \ref{thm-2} generalizes the correspondence (10) in \cite{Huang1995}, whose commutative diagram is valid for bundle valued $(0,q)$-forms.

The Chern connection $\nabla$ on pair $(X_0,E_0)$ is not integrable with respect to $(X_t,E_t)$ in general, but we have the following observation.

\begin{corollary}
\begin{align*}
\nabla^{0,1}_t\circ P_t(s)=P_t\circ\left(\bar{\partial}-\mathcal{L}_{\phi(t)}^{1,0}\right)s
\end{align*}
for any $E_0$-valued $(p,q)$-form $s\in A^0(X_0,E_0)$.
\end{corollary}

\begin{proof}
Since
\begin{align*}
\nabla^{0,1}_t=\bar{\partial}_t+\theta^{0,1}_t
\end{align*}
and
\begin{align*}
I^{p,q+1}_t\circ\theta^{0,1}_t=\theta^{0,1}_t\circ I^{p,q}_t,
\end{align*}
we have
\begin{align*}
\nabla^{0,1}_t\circ I^{p,q}_t=I^{p,q+1}_t\circ\nabla^{0,1}_t.
\end{align*}
Theorem \ref{thm-2} implies
\begin{align*}
\nabla^{0,1}_t\circ P_t(s)
&=\nabla^{0,1}_t\circ\left(I^{p,q}_t\right)^{-1}\circ e^{i_{\phi(t)}}\circ I^{p,q}_0(s)\\
&=\left(I^{p,q+1}_t\right)^{-1}\circ\nabla^{0,1}_t\circ e^{i_{\phi(t)}}\circ I^{p,q}_0(s)\\
&=\left(I^{p,q+1}_t\right)^{-1}\circ e^{i_{\phi(t)}}\circ
\left(\bar{\partial}_{E_0}+\nabla^{1,0}_{E_0}\circ i_{\phi(t)}\right)\circ I^{p,q}_0(s)\\
&=\left(I^{p,q+1}_t\right)^{-1}\circ e^{i_{\phi(t)}}\circ I^{p,q+1}_0\circ\left(\bar{\partial}_{E_0}-\mathcal{L}_{\phi(t)}^{1,0}\right)s
\end{align*}
\end{proof}

\begin{proposition}\label{pro-4}
\begin{align*}
\left(\bar{\partial}-\mathcal{L}_{\phi(t)}^{1,0}\right)\circ\left(\bar{\partial}-\mathcal{L}_{\phi(t)}^{1,0}\right)
=-i_{\phi(t)}\Theta.
\end{align*}
\end{proposition}

\begin{proof}
By Lemma 3.2 in \cite{Liu-Rao-Yang2014},
\begin{align*}
&\left(\bar{\partial}-\mathcal{L}_{\phi(t)}^{1,0}\right)\circ
\left(\bar{\partial}-\mathcal{L}_{\phi(t)}^{1,0}\right)\\
&=\left(\bar{\partial}-i_{\phi(t)}\circ\nabla^{1,0}_{E_0}+\nabla^{1,0}_{E_0}\circ i_{\phi(t)}\right)
\circ\left(\bar{\partial}-i_{\phi(t)}\circ\nabla^{1,0}_{E_0}+\nabla^{1,0}_{E_0}\circ i_{\phi(t)}\right)\\
&=-\bar{\partial}\circ i_{\phi(t)}\circ\nabla^{1,0}_{E_0}
+\bar{\partial}\circ\nabla^{1,0}_{E_0}\circ i_{\phi(t)}
-i_{\phi(t)}\circ\nabla^{1,0}_{E_0}\circ\bar{\partial}
+\nabla^{1,0}_{E_0}\circ i_{\phi(t)}\circ\bar{\partial}\\
&\quad+i_{\phi(t)}\circ\nabla^{1,0}_{E_0}\circ i_{\phi(t)}\circ\nabla^{1,0}_{E_0}
-\nabla^{1,0}_{E_0}\circ i_{\phi(t)}\circ i_{\phi(t)}\circ\nabla^{1,0}_{E_0}
+\nabla^{1,0}_{E_0}\circ i_{\phi(t)}\circ\nabla^{1,0}_{E_0}\circ i_{\phi(t)}\\
&=-i_{\bar{\partial}\phi(t)}\circ\nabla^{1,0}_{E_0}
+\left(\Theta\wedge\right)\circ i_{\phi(t)}
-i_{\phi(t)}\circ\left(\Theta\wedge\right)
-\nabla^{1,0}_{E_0}\circ i_{\bar{\partial}\phi(t)}\\
&\quad+\frac{1}{2}\left(i_{[\phi(t),\phi(t)]}\circ\nabla^{1,0}_{E_0}
+\nabla^{1,0}_{E_0}\circ i_{[\phi(t),\phi(t)]}\right)\\
&=-\left(i_{\phi(t)}\Theta\right)\wedge-\mathcal{L}^{1,0}_{\bar{\partial}\phi(t)-\frac{1}{2}[\phi(t),\phi(t)]}\\
&=-\left(i_{\phi(t)}\Theta\right)\wedge.
\end{align*}
The last equality is by Maurer-Cartan integrable equation $\bar{\partial}\phi(t)-\frac{1}{2}[\phi(t),\phi(t)]=0$.
\end{proof}

The Proposition \ref{pro-4} implies that
\begin{align*}
\nabla^{0,1}_t\circ\nabla^{0,1}_t\circ P_t=P_t\circ(-i_{\phi(t)}\Theta\wedge)=\left(P_t(-i_{\phi(t)}\Theta)\wedge\right)\circ P_t.
\end{align*}
Hence,
\begin{align*}
\nabla^{0,1}_t\circ\nabla^{0,1}_t=-P_t(i_{\phi(t)}\Theta)\wedge.
\end{align*}

\begin{corollary}\label{coro-1}
When we consider the line bundle case, i.e., each $E_t$ is holomorphic line bundle over $X_t$, then the second integrable equation in \cite{Liu-Tu2018}
\begin{align*}
(\bar{\partial}-L_{\phi(t)})\psi(t)-i_{\phi(t)}\Theta=0
\end{align*}
is solvable if and only if
\begin{align*}
\nabla^{0,1}_t\circ\nabla^{0,1}_t=\bar{\partial}_t\circ P_t\left(-\psi(t)\right)
\end{align*}
is solvable.
\end{corollary}
\begin{proof}
\begin{align*}
P_t(i_{\phi(t)}\Theta)=P_t\circ(\bar{\partial}-L_{\phi(t)})\left(\psi(t)\right)=\bar{\partial}_t\circ P_t\left(\psi(t)\right).
\end{align*}
\end{proof}

The above assertion implies Proposition 1.4 in \cite{Huang1995}, which tells that the line bundle is unobstructed along the base family if and only if $c_1(L_0)\in H^{1,1}(X_t)$.

\end{document}